\documentclass[a4paper,12pt]{article}

\usepackage{amsfonts}
\usepackage{amscd,color}
\usepackage{amsmath,amsfonts,amssymb,amscd}
\numberwithin{equation}{section}
\usepackage{indentfirst,graphicx,epsfig}
\usepackage{graphicx}
\input{epsf}
\usepackage{graphicx}
\usepackage{epstopdf}
\usepackage{caption}
\usepackage{subfigure}
\usepackage{mathrsfs}
\usepackage{ifpdf}
\usepackage{enumerate}
\usepackage{array}
\allowdisplaybreaks
\newtheorem{thm}{Theorem}[section]

\newtheorem{conjecture}[thm]{Conjecture}
\newtheorem{lem}[thm]{Lemma}

\newtheorem{claim}[thm]{Claim}

\newenvironment {proof} {\noindent{\textbf {Proof.}}}{\hfill$\Box$}
\newcommand{\ml}{l\kern-0.55mm\char39\kern-0.3mm}

\title{\textbf{Spectral radius of graphs of given size with forbidden a fan graph $F_6$}}
\author{{\small Jing Gao, Xueliang Li} \\
{\small  Center for Combinatorics and LPMC}\\
{\small Nankai University, Tianjin 300071, China}\\
{\small gjing1270@163.com, lxl@nankai.edu.cn}\\
}
\date{}
\begin{document}
\maketitle
\begin{abstract}
Let $F_k=K_1\vee P_{k-1}$ be the fan graph on $k$ vertices. A graph is said to be $F_k$-free if it does not contain $F_k$ as a subgraph. Yu et al. in [arXiv:2404.03423] conjectured that for $k\geq2$ and $m$ sufficiently large, if $G$ is an $F_{2k+1}$-free or $F_{2k+2}$-free graph, then $\lambda(G)\leq \frac{k-1+\sqrt{4m-k^2+1}}{2}$ and the equality holds if and only if $G\cong K_k\vee\left(\frac{m}{k}-\frac{k-1}{2}\right)K_1$. Recently, Li et al. in [arXiv:2409.15918]
showed that the above conjecture holds for $k\geq 3$. The only left case is for $k=2$, which corresponds to $F_5$ or $F_6$. Since the case of $F_5$ was solved by Yu et al. in [arXiv:2404.03423] and Zhang and Wang in [On the spectral radius of graphs without a gem, Discrete Math. 347 (2024) 114171]. So, one needs only to deal with the case of $F_6$. In this paper, we solve the only left case by determining the maximum spectral radius of $F_6$-free graphs with size $m\geq 88$, and the corresponding extremal graph.
\\[2mm]
\end{abstract}

\section{\bf Introduction}

Throughout this paper, we consider only simple and finite undirected graphs. Let $G$ be a simple graph with vertex set $V(G)$ and edge set $E(G)$. We use $|V(G)|=|G|$ and $e(G)$ to denote the order and the size of $G$, respectively.  Let $A(G)$ be the adjacency matrix of a graph $G$. Since $A(G)$ is real symmetric, its eigenvalues are real.  Hence, they can be ordered as $\lambda_1(G)\geq\cdots\geq\lambda_{|G|}(G)$ where $\lambda_1(G)$ is called the spectral radius of $G$ and also denoted by $\lambda(G)$. The neighborhood of a vertex $u\in V(G)$ is denoted by $N_G(u)$. Let $N_G[u] = N_G(u)\cup\{u\}$. The degree of a vertex $u$ in $G$ is denoted by $d_G(u)$. All the subscripts defined here will be omitted if it is clear from the context. As usual, $\Delta(G)$ stand for the maximum degree of $G$. For two subsets $X,Y\subseteq V (G)$, we use $e(X,Y)$ to denote the number of all edges of $G$ with one end vertex in $X$ and the other in $Y$. Particularly, $e(X,X)$ is simplified by $e(X)$. Denote $G[X]$ the subgraph of $G$ induced by $X$. For graph notation and concept undefined here, readers are referred to \cite{Bondy-Murty-2008}.

Let $K_n, C_n$, $P_n$ and $K_{1,n-1}$ be the complete graph, cycle, path and star on $n$ vertices, respectively. Let $K_{1,n-1}+e$ be the graph obtained from $K_{1,n-1}$ by adding one edge within its independent set. An $(a,b)$-double star, denoted by $D_{a,b}$, is the graph obtained by taking an edge and joining one of its end vertices with $a$ vertices and the other end vertex with $b$ vertices which are different from the $a$ vertices. The join of two disjoint graphs $G$ and $H$, denoted by $G\vee H$, is obtained from $G\cup H$ by adding all possible edges between $G$ and $H$.

A graph is said to be $F$-free if it does not contain a subgraph isomorphic to $F$. For a graph $F$ and an integer $m$, let $\mathcal{G}(m, F)$ be the set of $F$-free graphs of size $m$ without isolated vertices. An interesting spectral Tur\'{a}n type problem asks what is the maximum spectral radius of an $F$-free graph with given size $m$, which is also known as Brualdi-Hoffman-Tur\'{a}n type problem \cite{Brualdi-Hoffman-1985}. These extremal spectral graph problems have attracted wide attention recently, see \cite{Fang-You-2023,Li-Lu-2023,Lin-Ning-2021,Liu-Wang-2024,Lu-Lu-2024,Min-Lou-2022,Nikiforov-2002,Nikiforov-2006,Nikiforov-2009,Sun-Li-2023,Wang-2022,Zhai-Lin-2021, Zhai-Shu-2022}.

Let $F_k = K_1\vee P_{k-1}$ be the fan graph on $k$ vertices. Note that $F_3$ is a triangle and $F_4$ is a book on 4 vertices. Nosal \cite{Nosal-1970} showed that $\lambda(G)\leq\sqrt{m}$ for any $F_3$-free graph $G$ with size $m$. In 2021, Nikiforov \cite{Nikiforov} showed that if $G$ is a graph with $m$ edges and $\lambda(G)\geq\sqrt{m}$, then the maximum number
of triangles with a common edge in $G$ is greater than $\frac{1}{12}\sqrt[4]{m}$, unless $G$ is a complete bipartite graph with possibly some isolated vertices. From this, we obtain that the complete bipartite graphs attain the maximum spectral radius when the graphs are $F_4$-free. For $k=5$, Zhang and Wang \cite{Zhang-Wang-2024} and Yu et al. \cite{Yu-Li-2024} respectively considered the extremal problem on spectral radius for $F_5$-free graphs with size $m$. In \cite{Yu-Li-2024}, Yu et al. proposed  the following conjecture on spectral radius for $F_k$-free graphs with given size $m$.
\begin{conjecture}\label{conjecture}
Let $k\geq2$ be fixed and $m$ be sufficiently large. If $G$ is an $F_{2k+1}$-free or $F_{2k+2}$-free graph with $m$ edges, then
$$
\lambda(G)\leq \frac{k-1+\sqrt{4m-k^2+1}}{2},
$$
and equality holds if and only if $G\cong K_k\vee\left(\frac{m}{k}-\frac{k-1}{2}\right)K_1$.
\end{conjecture}
Recently, Li et al. \cite{Li-Zhao-2024} gave a unified approach to resolve Conjecture \ref{conjecture} for $k\geq3$. 
In order to resolve Conjecture \ref{conjecture} completely, one only need to consider $F_{6}$-free graphs. Motivated by it, we will show that Conjecture \ref{conjecture} holds for $F_{6}$-free graphs.

\begin{thm}\label{main theorem}
Let $G$ be an $F_6$-free graph with $m\geq88$ edges. Then $\lambda(G)\leq\frac{1+\sqrt{4m-3}}{2}$ and equality holds if and only if $G\cong K_2\vee\frac{m-1}{2}K_1$.
\end{thm}
\section{Preliminaries}
In this section, we present some preliminary results, which play an important role in the subsequent sections.
\begin{lem}\cite{Nikiforov-2002,Nosal-1970}\label{lem:Nikiforov}
If $G$ is a $K_3$-free graph with size $m$, then $\lambda(G)\leq\sqrt{m}$ and equality holds if and only if $G$ is a complete bipartite graph.
\end{lem}

For a connected graph $G$, by Perron-Frobenius theorem \cite{Cvetkovic-Rowlinson-2010}, we know that there exists a positive unit eigenvector corresponding to $\lambda(G)$, which is called the Perron vector of $G$.

\begin{lem}\cite{Wu-Xiao-2005}\label{lem:Wu-Xiao-2005}
Let $u$ and $v$ be two vertices of the connected graph $G$. Suppose $v_1,v_2,\ldots,v_s$ $(1\leq s\leq d_G(v))$ are some vertices of $N_G(v)\setminus N_G(u)$ and $\mathbf{x}$ is the Perron vector of $G$ with $x_w$ corresponding to the vertex $w\in V(G)$. Let $G'=G-\{vv_i|1\leq i\leq s\}+\{uv_i|1\leq i\leq s\}$. If $x_u\geq x_v$, then $\lambda(G')>\lambda(G)$.
\end{lem}

A cut vertex of a graph is a vertex whose deletion increases the number of components. A graph is called 2-connected, if it is a connected graph without cut vertices. Let $\mathbf{x}$ be the Perron vector of $G$ with coordinate $x_u$ corresponding to the vertex $u\in V(G)$ and $u^\ast$ be a vertex satisfying $x_{u^\ast}=\max\{x_u|u\in V(G)\}$, which is said to be an extremal vertex.

\begin{lem}\cite{Zhai-Lin-2021}\label{lem:Zhai-Lin-2021}
Let $G$ be a graph in $\mathcal{G}(m,F)$ with the maximum spectral radius. If $F$ is a 2-connected graph and $u^\ast$ is an extremal vertex of $G$, then $G$ is connected and $d(u)\geq2$ for any $u\in V(G)\setminus N[u^\ast]$.
\end{lem}

The following result is mentioned in \cite{Li-Zhao-2024} which is easy to get.
\begin{lem}\cite{Li-Zhao-2024}\label{lem:Li-Zhao-2024}
Let $G$ be a graph in $\mathcal{G}(m,F_k)$. Then for all $u\in V(G)$, the graph $G[N(u)]$ is $P_{k-1}$-free.
\end{lem}
\section{Proof of Theorem \ref{main theorem}}
Let $G^\ast$ be a graph in $\mathcal{G}(m,F_6)$ with the maximum spectral radius. By Lemma \ref{lem:Zhai-Lin-2021}, we have $G^\ast$ is connected. Let $\lambda=\lambda(G^\ast)$ and $\mathbf{x}$ be the Perron vector of $G^\ast$ with coordinate $x_v$ corresponding to the vertex $v\in V(G^\ast)$. Assume that $u^\ast$ is an extremal vertex of $G^\ast$. Set $U=N_{G^\ast}(u^\ast)$ and $W=V(G^\ast)\setminus N_{G^\ast}[u^\ast]$.  Let $U_0$ be the isolated vertices of the induced subgraph $G^\ast[U]$, and $U_+=U\setminus U_0$ be the vertices of $U$ with degree at least one in $G^\ast[U]$. Let $W_H=N_W(V(H))$ for any subset $H$ of $G^\ast[U]$ and $W_0=\{w|d_W(w)=0\}$.

Note that $\lambda\left(K_2\vee\frac{m-1}{2}K_1\right)=\frac{1+\sqrt{4m-3}}{2}$ and $K_2\vee\frac{m-1}{2}K_1$ is $F_6$-free, we have
$$
\lambda(G^\ast)\geq\lambda\left(K_2\vee\frac{m-1}{2}K_1\right)=\frac{1+\sqrt{4m-3}}{2}.
$$
Hence, $\lambda^2-\lambda\geq m-1$. Furthermore, Since $m\geq88$, we can get $\lambda=\lambda(G^\ast)>\frac{49}{5}$.
Since $\lambda(G^\ast)\mathbf{x}=A(G^\ast)\mathbf{x}$, we have
$$
\lambda x_{u^\ast}=\sum_{u\in U_+}x_u+\sum_{u\in U_0}x_u.
$$
Furthermore, $\mathbf{x}$ is also an eigenvector of $A^2(G^\ast)$ corresponding to $\lambda^2(G^\ast)$. It follows that
\begin{align*}
\lambda^2x_{u^\ast}&=|U|x_{u^\ast}+\sum_{u\in U_+}d_U(u)x_u+\sum_{w\in W}d_U(w)x_w.
\end{align*}
Therefore,
\begin{align*}
(\lambda^2-\lambda)x_{u^\ast}=|U|x_{u^\ast}+\sum_{u\in U_+}(d_U(u)-1)x_u+\sum_{w\in W}d_U(w)x_w-\sum_{u\in U_0}x_u.
\end{align*}
Since $\lambda^2-\lambda\geq m-1=|U|+e(U_+)+e(U,W)+e(W)-1$, we have
$$
\sum_{u\in U_+}(d_U(u)-1)x_u+\sum_{w\in W}d_U(w)x_w\geq\left(e(U_+)+e(U,W)+e(W)+\sum_{u\in U_0}\frac{x_u}{x_{u^\ast}}-1\right)x_{u^\ast}.
$$
That is,
\begin{align*}
\sum_{u\in U_+}(d_U(u)-1)\frac{x_u}{x_{u^\ast}}+\sum_{w\in W}d_U(w)\frac{x_w}{x_{u^\ast}}\geq e(U_+)+e(U,W)+e(W)+\sum_{u\in U_0}\frac{x_u}{x_{u^\ast}}-1.\tag{1}
\end{align*}

\begin{lem}\label{edge number of U}
$e(U)\geq4$.
\end{lem}
\begin{proof}
Since $\lambda(G^\ast)\geq\frac{1+\sqrt{4m-3}}{2}>\sqrt{m+3}$ when $m\geq88$, we have
\begin{align*}
m+3&<\lambda^2=|U|+\sum_{u\in U_+}d_U(u)\frac{x_u}{x_{u^\ast}}+\sum_{w\in W}d_U(w)\frac{x_w}{x_{u^\ast}}\\
&\leq|U|+\sum_{u\in U_+}d_U(u)+\sum_{w\in W}d_U(w)\\
&=|U|+2e(U)+e(U,W).
\end{align*}
Note that $m=|U|+e(U)+e(U,W)+e(W)$. It follows that $e(U)> e(W)+3\geq3$. Hence $e(U)\geq4$.
\end{proof}

By Lemma \ref{edge number of U}, there exists at least one non-trivial component in $G^\ast[U]$. Let $\mathcal{H}$ be the set of all non-trivial components in $G^\ast[U]$. For each non-trivial component $H$ of $\mathcal{H}$, we denote $\gamma(H):=\sum_{u\in V(H)}(d_H(u)-1)\frac{x_u}{x_{u^\ast}}-e(H)$. Clearly, by (1) we have
\begin{align*}
e(W)&\leq\sum_{H\in\mathcal{H}}\gamma(H)+\sum_{w\in W}d_U(w)\frac{x_w}{x_{u^\ast}}-e(U,W)-\sum_{u\in U_0}\frac{x_u}{x_{u^\ast}}+1\\
&\leq\sum_{H\in\mathcal{H}}\gamma(H)-\sum_{u\in U_0}\frac{x_u}{x_{u^\ast}}+1\tag{2}
\end{align*}
with equality if and only if $\lambda^2-\lambda= m-1$ and $x_w=x_{u^\ast}$ for any $w\in W$ with $d_U(w)\geq 1$.

Since $G^\ast$ is $F_6$-free, by Lemma \ref{lem:Li-Zhao-2024}, we know that $G^\ast[U]$ does not contain path $P_5$. Similarly to the proof of Lemma 4.5 in \cite{Liu-Wang-2024}, we have the following result.
\begin{lem}\label{component cases}
Let $G^\ast$ be an $F_6$-free graph with $u\in V(G^\ast)$ and $H$ be a component of $G^\ast[N(u)]$. Then $H$ is one of the following cases:
\begin{enumerate}[(i)]
\item a star $K_{1,r}$ for $r\geq0$, where $K_{1,0}$ is an isolated vertex;
\item a double star $D_{a,b}$ for $a,b\geq1$;
\item $K_{1,r}+e$ where $r\geq2$ and $K_{1,2}+e$ is a triangle;
\item $C_4,K_4-e$ or $K_4$.
\end{enumerate}
\end{lem}

Next we provide some upper bounds on $\gamma(H)$ where $H$ is a non-trivial component of $G^\ast[U]$.
\begin{lem}\label{upperbound}
Let $H$ be a non-trivial component of $G^\ast[U]$, then
\begin{align*}
\gamma(H)\leq
\begin{cases}
-1, \text{ if $H\cong K_{1,r}$ or $D_{a,b}$ where $r,a,b\geq1$},\\
0, \text{ if $H\cong K_{1,r}+e$ where $r\geq2$ or $C_4$ },\\
1, \text{ if $H\cong K_4-e$},\\
2, \text{ if $H\cong K_4$}.
\end{cases}
\end{align*}
\end{lem}

\begin{proof}
Since $\frac{x_u}{x_{u^\ast}}\leq1$ and $d_H(u)\geq1$ for any $u\in V(H)$, we have
\begin{align*}
\gamma(H)=\sum_{u\in V(H)}(d_H(u)-1)\frac{x_u}{x_{u^\ast}}-e(H)\leq\sum_{u\in V(H)}(d_H(u)-1)-e(H)= e(H)-|H|.
\end{align*}
If $H\cong K_{1,r}$ or $D_{a,b}$, then $e(H)-|H|=-1$. When $H\cong K_{1,r}+e$ or $C_4$, we have $e(H)-|H|=0$. For $H\cong K_4-e$, we get $e(H)-|H|=1$. If $H\cong K_4$, it follows that $e(H)-|H|=2$. Hence, the lemma holds.
\end{proof}

In fact, we can give a tighter upper bound about $\gamma(H)$.
\begin{lem}\label{component lemma}
For every non-trivial component $H$ in $G^\ast[U]$, we have $\gamma(H)\leq0$.
\end{lem}

\begin{proof}
Suppose to the contrary that there exists at least one component $H\in\mathcal{H}$ satisfying $\gamma(H)>0$. Let $\widetilde{\mathcal{H}}=\{H\in\mathcal{H}:\gamma(H)>0\}$. By Lemma \ref{upperbound}, we know that $H$ is $K_4$ or $K_4-e$ for any component $H\in\widetilde{\mathcal{H}}$. We have the following claims.

\begin{claim}\label{bound of eigenvector}
$\sum_{u\in V(H)}x_u>\frac{5}{2}x_{u^\ast}$ for any $H\in\widetilde{\mathcal{H}}$.
\end{claim}

\begin{proof}
Suppose that $\sum_{u\in V(H_0)}x_u\leq\frac{5}{2}x_{u^\ast}$ for some $H_0\in\widetilde{\mathcal{H}}$. Then
\begin{align*}
\gamma(H_0)&=\sum_{u\in V(H_0)}(d_{H_0}(u)-1)\frac{x_u}{x_{u^\ast}}-e(H_0)\\
&\leq(\Delta(H_0)-1)\sum_{u\in V(H_0)}\frac{x_u}{x_{u^\ast}}-e(H_0)\\
&\leq2\sum_{u\in V(H_0)}\frac{x_u}{x_{u^\ast}}-5\\
&\leq0,
\end{align*}
which contradicts $H_0\in\widetilde{\mathcal{H}}$. Therefore, $\sum_{u\in V(H)}x_u>\frac{5}{2}x_{u^\ast}$ for any $H\in\widetilde{\mathcal{H}}$.
\end{proof}

\begin{claim}\label{component number}
The number $|\widetilde{\mathcal{H}}|$ of members in $\widetilde{\mathcal{H}}$ satisfying $|\widetilde{\mathcal{H}}|<\frac{2}{5}\lambda+\frac{7}{10}$.
\end{claim}

\begin{proof}
Suppose $|\widetilde{\mathcal{H}}|\geq\frac{2}{5}\lambda+\frac{7}{10}$. For every $H\in\widetilde{\mathcal{H}}$, we have
\begin{align*}
\lambda\sum_{u\in V(H)}x_u&=\sum_{u\in V(H)}\left(x_{u^\ast}+\sum_{v\in N_H(u)}x_v+\sum_{w\in N_W(u)}x_w\right)\\
&=4x_{u^\ast}+\sum_{u\in V(H)}\sum_{v\in N_H(u)}x_v+\sum_{u\in V(H)}\sum_{w\in N_W(u)}x_w\\
&\leq4x_{u^\ast}+\sum_{u\in V(H)}d_H(u)x_u+e(H,W)x_{u^\ast}\\
&=4x_{u^\ast}+\gamma(H)x_{u^\ast}+e(H)x_{u^\ast}+\sum_{u\in V(H)}x_u+e(H,W)x_{u^\ast}.
\end{align*}
Note that $H$ is $K_4$ or $K_4-e$, by Lemma \ref{upperbound}, we have $\gamma(H)\leq2$. Furthermore, $e(H)\leq6$. Therefore,
\begin{align*}
(\lambda-1)\sum_{u\in V(H)}x_u&\leq4x_{u^\ast}+\gamma(H)x_{u^\ast}+e(H)x_{u^\ast}+e(H,W)x_{u^\ast}\\
&\leq4x_{u^\ast}+2x_{u^\ast}+6x_{u^\ast}+e(H,W)x_{u^\ast}\\
&=12x_{u^\ast}+e(H,W)x_{u^\ast}.
\end{align*}
Hence, by Claim \ref{bound of eigenvector}, we obtain
\begin{align*}
e(H,W)x_{u^\ast}&\geq(\lambda-1)\sum_{u\in V(H)}x_u-12x_{u^\ast}\\
&>\frac{5(\lambda-1)}{2}x_{u^\ast}-12x_{u^\ast}\\
&=\frac{5\lambda-29}{2}x_{u^\ast}.
\end{align*}
It follows that $e(H,W)>\frac{5\lambda-29}{2}$. Since
\begin{align*}
m&\geq |U|+\sum_{H\in\widetilde{\mathcal{H}}}e(H)+\sum_{H\in\widetilde{\mathcal{H}}}e(H,W)\\
&>4|\widetilde{\mathcal{H}}|+5|\widetilde{\mathcal{H}}|+\left(\frac{5\lambda-29}{2}\right)|\widetilde{\mathcal{H}}|\\
&=\left(\frac{5\lambda-11}{2}\right)|\widetilde{\mathcal{H}}|\\
&\geq\left(\frac{5\lambda-11}{2}\right)\left(\frac{2}{5}\lambda+\frac{7}{10}\right)\\
&=\lambda^2-\frac{9}{20}\lambda-\frac{77}{20}.
\end{align*}
Since $\lambda>\frac{49}{5}$, we have $m>\lambda^2-\frac{9}{20}\lambda-\frac{77}{20}>\lambda^2-\lambda+1$, which contradicts $\lambda\geq\frac{1+\sqrt{4m-3}}{2}$. Therefore, $|\widetilde{\mathcal{H}}|<\frac{2}{5}\lambda+\frac{7}{10}$.
\end{proof}

\begin{claim}\label{degree of W0}
If $W_0\neq\emptyset$, then $x_w\leq\frac{2\lambda-3}{2\lambda}x_{u^\ast}$ for any $w\in W_0$.
\end{claim}

\begin{proof}
Suppose to the contrary that there exists $w_0\in W_0$ such that $x_{w_0}>\frac{2\lambda-3}{2\lambda}x_{u^\ast}$. For every $H\in\widetilde{\mathcal{H}}$, if $w_0\notin N_{G^\ast}(u)$ for every $u\in V(H)$, then $\lambda x_{w_0}\leq\lambda x_{u^\ast}-\sum_{u\in V(H)}x_u$. It follows that $\sum_{u\in V(H)}x_u\leq\lambda x_{u^\ast}-\lambda x_{w_0}<\lambda x_{u^\ast}-\frac{2\lambda-3}{2}x_{u^\ast}=\frac{3}{2}x_{u^\ast}$. Hence $\gamma(H)\leq2\sum_{u\in V(H)}\frac{x_u}{x_{u^\ast}}-e(H)<3-5<0$, contradicting $H\in\widetilde{\mathcal{H}}$. Next, we consider that there exists $H'\in\widetilde{\mathcal{H}}$ satisfying $w_0\in N_{G^\ast}(u)$ for some $u\in V(H')$. Let $V(H')=\{u_1,u_2,u_3,u_4\}$ and $u_1u_2u_3u_4u_1$ is a cycle in $H'$. Without loss of generality, assume $w_0\in N_W(u_1)$. If $H'\cong K_4$, then $N_U(w_0)\cap V(H')=\{u_1\}$. Otherwise, $G^\ast[u^\ast,u_1,u_2,u_3,u_4,w_0]$ contains an $F_6$. Therefore, $\lambda x_{w_0}\leq\lambda x_{u^\ast}-\sum_{i=2}^4x_{u_i}$. It follows that $\sum_{i=2}^4x_{u_i}\leq\lambda x_{u^\ast}-\lambda x_{w_0}<\frac{3}{2}x_{u^\ast}$. Thus, $\gamma(H')\leq2\sum_{u\in V(H')}\frac{x_u}{x_{u^\ast}}-e(H')\leq2+2\sum_{i=2}^4\frac{x_{u_i}}{x_{u^\ast}}-6<2+3-6<0$, contradicting $H'\in\widetilde{\mathcal{H}}$. When $H'\cong K_4-e$, if $d_{H'}(u_1)=2$, then $u_2,u_4\notin N_{G^\ast}(w_0)$. Otherwise, $G^\ast$ contains an $F_6$. So $\lambda x_{w_0}\leq\lambda x_{u^\ast}-(x_{u_2}+x_{u_4})$. That is, $x_{u_2}+x_{u_4}\leq\lambda x_{u^\ast}-\lambda x_{w_0}<\frac{3}{2}x_{u^\ast}$. Hence, $\gamma(H')=\frac{x_{u_1}}{x_{u^\ast}}+2\cdot\frac{x_{u_2}}{x_{u^\ast}}+\frac{x_{u_3}}{x_{u^\ast}}+2\cdot\frac{x_{u_4}}{x_{u^\ast}}-5<2+2\cdot\frac{3}{2}-5=0$, a contradiction. If $d_{H'}(u_1)=3$, then $u_2,u_3,u_4\notin N_{G^\ast}(w_0)$. By $\lambda x_{w_0}\leq\lambda x_{u^\ast}-(x_{u_2}+x_{u_3}+x_{u_4})$, we obtain $x_{u_2}+x_{u_3}+x_{u_4}\leq\lambda x_{u^\ast}-\lambda x_{w_0}<\frac{3}{2}x_{u^\ast}$. Hence, $\gamma(H')=2\cdot\frac{x_{u_1}}{x_{u^\ast}}+\frac{x_{u_2}}{x_{u^\ast}}+2\cdot\frac{x_{u_3}}{x_{u^\ast}}+\frac{x_{u_4}}{x_{u^\ast}}-5<3+\frac{3}{2}-5<0$, a contradiction. Therefore,  $x_w\leq\frac{2\lambda-3}{2\lambda}x_{u^\ast}$ for any $w\in W_0$.
\end{proof}

Now, we come back to prove Lemma \ref{component lemma}. Choose a component $H\in\widetilde{\mathcal{H}}$. We proceed by distinguishing the following two cases.

{\bf Case 1. }$e(W)\leq3$.

If $W_H=\emptyset$, then $d_{G^\ast}(u)\leq4$ for every $u\in V(H)$. Therefore, $\lambda x_u=\sum_{v\in N(u)}x_v\leq4x_{u^\ast}$. That is, $x_u\leq\frac{4}{\lambda}x_{u^\ast}$. By $d_H(u)\leq3$ for every $u\in V(H)$ and $e(H)\geq5$, we have $\gamma(H)\leq2\sum_{u\in V(H)}\frac{x_u}{x_{u^\ast}}-5\leq2\cdot4\cdot\frac{4}{\lambda}-5=\frac{32}{\lambda}-5$. Note that $\lambda>\frac{49}{5}$, then $\gamma(H)<0$, which contradicts $H\in\widetilde{\mathcal{H}}$. Now we consider $W_H\neq\emptyset$. If $H\cong K_4$, we have $|N_H(w)\cap V(H)|=1$ for any $w\in W_H$. Otherwise, $G^\ast$ contains an $F_6$, a contradiction. Let $x_{u_1}=\max\{x_u|u\in V(H)\}$. It follows that $N_H(w)\cap V(H)=\{u_1\}$. Otherwise, suppose $N_H(w_0)\cap V(H)=\{u_2\}$ for some $w_0\in W_H$. By $e(W)\leq3$, we get $G'=G^\ast-w_0u_2+w_0u_1$ is $F_6$-free. According to $x_{u_1}\geq x_{u_2}$ and Lemma \ref{lem:Wu-Xiao-2005}, we get $\lambda(G')>\lambda(G^\ast)$, which contradicts the maximality of $G^\ast$. Thus, $N_W(u)=\emptyset$ for any $u\in V(H)\setminus u_1$. That is, $d_{G^\ast}(u)=4$ for every $u\in V(H)\setminus u_1$. So we obtain $\lambda x_u\leq4x_{u^\ast}$ for every $u\in V(H)\setminus u_1$. Therefore, by $\lambda>\frac{49}{5}$, we have $\gamma(H)\leq2+6\cdot\frac{4}{\lambda}-6<0$, a contradiction. If $H\cong K_4-e$, suppose $d_H(u_1)=d_H(u_2)=3$ and $x_{u_1}\geq x_{u_2}$. If $N_W(u_1)\cup N_W(u_2)=\emptyset$, then $d_{G^\ast}(u_1)=d_{G^\ast}(u_2)=4$. It follows that $\lambda x_{u_i}\leq4x_{u^\ast}$ for $i=1,2$. Then $\gamma(H)\leq2\sum_{i=1}^2\frac{x_{u_i}}{x_{u^\ast}}+\sum_{u\in V(H)\setminus\{u_1,u_2\}}\frac{x_{u_i}}{x_{u^\ast}}-5\leq4\cdot\frac{4}{\lambda}+2-5<0$, a contradiction. If $N_W(u_1)\cup N_W(u_2)\neq\emptyset$,  since $G^\ast$ is $F_6$-free, we have $|N_G(w)\cap\{u_1,u_2\}|=1$ and $|N_G(w)\cap (V(H)\setminus\{u_1,u_2\})|=0$ for every $w\in N_W(u_1)\cup N_W(u_2)$. Recall that $x_{u_1}\geq x_{u_2}$. It follows that $N_G(w)\cap\{u_1,u_2\}=\{u_1\}$. Otherwise, there is a vertex $w_0\in N_W(u_2)$. By $e(W)\leq3$, we have $G'=G^\ast-w_0u_2+w_0u_1$ is $F_6$-free. According to Lemma \ref{lem:Wu-Xiao-2005}, we get $\lambda(G')>\lambda(G^\ast)$, which contradicts the maximality of $G^\ast$. Thus, $N_W(u_2)=\emptyset$. That is, $d_{G^\ast}(u_2)=4$. So $\lambda x_{u_2}\leq4x_{u^\ast}$. Hence, by $\lambda>\frac{49}{5}$, we get $\gamma(H)\leq2\sum_{i=1}^2\frac{x_{u_i}}{x_{u^\ast}}+\sum_{u\in V(H)\setminus\{u_1,u_2\}}\frac{x_{u_i}}{x_{u^\ast}}-5\leq2+2\cdot\frac{4}{\lambda}+2-5<0$, a contradiction.

{\bf Case 2. }$e(W)\geq4$.

Let $V(H)=\{u_1,u_2,u_3,u_4\}$. If $H\cong K_4$, then $N_W(u_1)\cap N_W(u_2)\cap N_W(u_3)\cap N_W(u_4)=\emptyset$. Otherwise, $G^\ast$ contains an $F_6$. Therefore, $\sum_{i=1}^4\sum_{w\in N_{W_{H}}(u_i)}x_w=\sum_{w\in W_0\cap W_H}x_w+\sum_{w\in W_H\setminus W_0}x_w$. By Claim \ref{degree of W0}, we have $\sum_{w\in W_0\cap W_H}x_w\leq \frac{(2\lambda-3)e(H,W_0)}{2\lambda}x_{u^\ast}$. Since $d_W(w)\geq1$ for $w\in W_H\setminus W_0$, we have $\sum_{w\in W_H\setminus W_0}x_w\leq\sum_{w\in W_H\setminus W_0}d_W(w)x_w\leq2e(W)x_{u^\ast}$. Since
\begin{align*}
\begin{cases}
\lambda x_{u_1}= x_{u_2}+x_{u_3}+x_{u_4}+x_{u^\ast}+\sum_{w\in N_{W_{H}}(u_1)}x_w,\\
\lambda x_{u_2}= x_{u_1}+x_{u_3}+x_{u_4}+x_{u^\ast}+\sum_{w\in N_{W_{H}}(u_2)}x_w,\\
\lambda x_{u_3}= x_{u_1}+x_{u_2}+x_{u_4}+x_{u^\ast}+\sum_{w\in N_{W_{H}}(u_3)}x_w,\\
\lambda x_{u_4}= x_{u_1}+x_{u_2}+x_{u_3}+x_{u^\ast}+\sum_{w\in N_{W_{H}}(u_4)}x_w,
\end{cases}
\end{align*}
we have
\begin{align*}
(\lambda-3)(x_{u_1}+x_{u_2}+x_{u_3}+x_{u_4})=4x_{u^\ast}+\sum_{i=1}^4\sum_{w\in N_{W_{H}}(u_i)}x_w.
\end{align*}
That is,
\begin{align*}
x_{u_1}+x_{u_2}+x_{u_3}+x_{u_4}&=\frac{4x_{u^\ast}+\sum_{i=1}^4\sum_{w\in N_{W_{H}}(u_i)}x_w}{\lambda-3}\\
&\leq\frac{4}{\lambda-3}x_{u^\ast}+\frac{(2\lambda-3)e(H,W_0)}{2\lambda(\lambda-3)}x_{u^\ast}+\frac{2e(W)}{\lambda-3}x_{u^\ast}.
\end{align*}
Thus,
\begin{align*}
\sum_{u\in V(H)}(d_H(u)-1)\frac{x_u}{x_{u^\ast}}&=\frac{2(x_{u_1}+x_{u_2}+x_{u_3}+x_{u_4})}{x_{u^\ast}}\\
&=\frac{x_{u_1}+x_{u_2}+x_{u_3}+x_{u_4}}{x_{u^\ast}}+\frac{x_{u_1}+x_{u_2}+x_{u_3}+x_{u_4}}{x_{u^\ast}}\\
&\leq4+\frac{4}{\lambda-3}+\frac{(2\lambda-3)e(H,W_0)}{2\lambda(\lambda-3)}+\frac{2e(W)}{\lambda-3}\\
&=\left(4+\frac{4}{\lambda-3}+\frac{\frac{2\lambda-5}{2\lambda-4}e(W)}{\lambda-3}\right)+\frac{(2\lambda-3)e(H,W_0)}{2\lambda(\lambda-3)}+\frac{\frac{2\lambda-3}{2\lambda-4}e(W)}{\lambda-3}.
\end{align*}
By (2), Lemma \ref{upperbound} and Claim \ref{component number}, we have $e(W)\leq2|\widetilde{\mathcal{H}}|+1<\frac{4}{5}\lambda+\frac{12}{5}$. This implies that $\frac{4}{\lambda-3}+\frac{\frac{2\lambda-5}{2\lambda-4}e(W)}{\lambda-3}<\frac{8\lambda^2+44\lambda-140}{5(\lambda-3)(2\lambda-4)}$. Since $\lambda>\frac{49}{5}$, we have $\frac{4}{\lambda-3}+\frac{\frac{2\lambda-5}{2\lambda-4}e(W)}{\lambda-3}<2$. Hence,
\begin{align*}
\sum_{u\in V(H)}(d_H(u)-1)\frac{x_u}{x_{u^\ast}}&<6+\frac{(2\lambda-3)e(H,W_0)}{2\lambda(\lambda-3)}+\frac{\frac{2\lambda-3}{2\lambda-4}e(W)}{\lambda-3}\\
&=e(H)+\frac{(2\lambda-3)e(H,W_0)}{2\lambda(\lambda-3)}+\frac{\frac{2\lambda-3}{2\lambda-4}e(W)}{\lambda-3}.
\end{align*}

If $H\cong K_4-e$, let $d_H(u_2)=d_H(u_4)=3$. Then $N_W(u_1)\cap N_W(u_2)\cap N_W(u_4)=\emptyset$. Hence, $\sum_{i\in\{1,2,4\}}\sum_{w\in N_{W_{H}}(u_i)}x_w=\sum_{w\in W_0\cap (\cup_{i\in\{1,2,4\}}N_W(u_i))}x_w+\sum_{w\in (\cup_{i\in\{1,2,4\}}N_W(u_i))\setminus W_0}x_w$. By Claim \ref{degree of W0}, we have $\sum_{w\in W_0\cap (\cup_{i\in\{1,2,4\}}N_W(u_i))}x_w\leq \frac{(2\lambda-3)e(H,W_0)}{2\lambda}x_{u^\ast}$. Since $d_W(w)\geq1$ for $w\in (\cup_{i\in\{1,2,4\}}N_W(u_i))\setminus W_0$, we have $\sum_{w\in (\cup_{i\in\{1,2,4\}}N_W(u_i))\setminus W_0}x_w\leq\sum_{w\in W_H\setminus W_0}d_W(w)x_w\leq2e(W)x_{u^\ast}$. Since
\begin{align*}
\begin{cases}
\lambda x_{u_1}= x_{u_2}+x_{u_4}+x_{u^\ast}+\sum_{w\in N_{W_{H}}(u_1)}x_w,\\
\lambda x_{u_2}= x_{u_1}+x_{u_3}+x_{u_4}+x_{u^\ast}+\sum_{w\in N_{W_{H}}(u_2)}x_w,\\
\lambda x_{u_4}= x_{u_1}+x_{u_2}+x_{u_3}+x_{u^\ast}+\sum_{w\in N_{W_{H}}(u_4)}x_w,
\end{cases}
\end{align*}
we have
\begin{align*}
\lambda(x_{u_1}+x_{u_2}+x_{u_4})=2(x_{u_1}+x_{u_2}+x_{u_4})+2x_{u_3}+3x_{u^\ast}+\sum_{i\in\{1,2,4\}}\sum_{w\in N_{W_{H}}(u_i)}x_w.
\end{align*}
That is,
\begin{align*}
x_{u_1}+x_{u_2}+x_{u_4}&\leq\frac{5x_{u^\ast}+\sum_{i\in\{1,2,4\}}\sum_{w\in N_{W_{H}}(u_i)}x_w}{\lambda-2}\\
&\leq\frac{5}{\lambda-2}x_{u^\ast}+\frac{(2\lambda-3)e(H,W_0)}{2\lambda(\lambda-2)}x_{u^\ast}+\frac{2e(W)}{\lambda-2}x_{u^\ast}.
\end{align*}
Thus,
\begin{align*}
\sum_{u\in V(H)}(d_H(u)-1)\frac{x_u}{x_{u^\ast}}&=\frac{x_{u_1}+2x_{u_2}+x_{u_3}+2x_{u_4}}{x_{u^\ast}}\\
&=\frac{x_{u_2}+x_{u_3}+x_{u_4}}{x_{u^\ast}}+\frac{x_{u_1}+x_{u_2}+x_{u_4}}{x_{u^\ast}}\\
&\leq3+\frac{5}{\lambda-2}+\frac{(2\lambda-3)e(H,W_0)}{2\lambda(\lambda-2)}+\frac{2e(W)}{\lambda-2}\\
&=\left(3+\frac{5}{\lambda-2}+\frac{e(W)}{\lambda-2}\right)+\frac{(2\lambda-3)e(H,W_0)}{2\lambda(\lambda-2)}+\frac{e(W)}{\lambda-2}.
\end{align*}
By (2), Lemma \ref{upperbound} and Claim \ref{component number}, we have $e(W)\leq2|\widetilde{\mathcal{H}}|+1<\frac{4}{5}\lambda+\frac{12}{5}$. This implies that $\frac{5}{\lambda-2}+\frac{e(W)}{\lambda-2}<\frac{4\lambda+37}{5(\lambda-2)}$. Since $\lambda>\frac{49}{5}$, we have $\frac{5}{\lambda-2}+\frac{e(W)}{\lambda-2}<2$. Hence,
\begin{align*}
\sum_{u\in V(H)}(d_H(u)-1)\frac{x_u}{x_{u^\ast}}&<5+\frac{(2\lambda-3)e(H,W_0)}{2\lambda(\lambda-2)}+\frac{e(W)}{\lambda-2}\\
&=e(H)+\frac{(2\lambda-3)e(H,W_0)}{2\lambda(\lambda-2)}+\frac{e(W)}{\lambda-2}.
\end{align*}
By $\lambda>\frac{49}{5}$, we get $\frac{1}{\lambda-2}\leq\frac{1}{\lambda-3}\cdot\frac{2\lambda-3}{2\lambda-4}$. Therefore, $\sum_{u\in V(H)}(d_H(u)-1)\frac{x_u}{x_{u^\ast}}<e(H)+\frac{(2\lambda-3)e(H,W_0)}{2\lambda(\lambda-3)}+\frac{\frac{2\lambda-3}{2\lambda-4}e(W)}{\lambda-3}$ for any $H\in\widetilde{\mathcal{H}}$. For any $H\in\mathcal{H}\setminus\widetilde{\mathcal{H}}$, since $\gamma(H)\leq0$, we have $\sum_{u\in V(H)}(d_H(u)-1)\frac{x_u}{x_{u^\ast}}\leq e(H)$. Thus, by $\lambda>\frac{49}{5}$, we obtain
\begin{align*}
&\sum_{u\in U_+}(d_U(u)-1)\frac{x_u}{x_{u^\ast}}+\sum_{w\in W}d_U(w)\frac{x_w}{x_{u^\ast}}\\
&=\sum_{H\in \widetilde{\mathcal{H}}}\sum_{u\in V(H)}(d_H(u)-1)\frac{x_u}{x_{u^\ast}}+\sum_{H\in \mathcal{H}\setminus\widetilde{\mathcal{H}}}\sum_{u\in V(H)}(d_H(u)-1)\frac{x_u}{x_{u^\ast}}+\sum_{w\in W_0}d_U(w)\frac{x_w}{x_{u^\ast}}+\sum_{w\in W\setminus W_0}d_U(w)\frac{x_w}{x_{u^\ast}}\\
&<\sum_{H\in \widetilde{\mathcal{H}}}\left(e(H)+\frac{(2\lambda-3)e(H,W_0)}{2\lambda(\lambda-3)}+\frac{\frac{2\lambda-3}{2\lambda-4}e(W)}{\lambda-3}\right)+\sum_{H\in \mathcal{H}\setminus\widetilde{\mathcal{H}}}e(H)+\sum_{w\in W_0}d_U(w)\cdot\frac{2\lambda-3}{2\lambda}\\
&+\sum_{w\in W\setminus W_0}d_U(w)\\
&< e(U_+)+\frac{(2\lambda-3)e(U,W_0)}{2\lambda(\lambda-3)}+\frac{\frac{2\lambda-3}{2\lambda-4}e(W)}{\lambda-3}\cdot\left(\frac{2}{5}\lambda+\frac{7}{10}\right)+\frac{(2\lambda-3)e(U,W_0)}{2\lambda}+e(U,W\setminus W_0)\\
&=e(U_+)+\left(\frac{2\lambda-3}{2\lambda(\lambda-3)}+\frac{2\lambda-3}{2\lambda}\right)e(U,W_0)+e(U,W\setminus W_0)+\frac{1}{\lambda-3}\cdot\frac{2\lambda-3}{2\lambda-4}\cdot\left(\frac{2}{5}\lambda+\frac{7}{10}\right)e(W)\\
&<e(U_+)+e(U,W_0)+e(U,W\setminus W_0)+\frac{3}{4}e(W)\\
&\leq e(U_+)+e(U,W)+e(W)-1,
\end{align*}
which contradicts (1). This completes the proof of Lemma \ref{component lemma}.
\end{proof}

By (2) and Lemma \ref{component lemma}, we know that $e(W)\leq1$.

\begin{lem}\label{no K4 lemma}
For any $H\in\mathcal{H}$, we have $H\ncong K_4$.
\end{lem}

\begin{proof}
Suppose to the contrary that there exists a component $H_0\in\mathcal{H}$ satisfying $H_0\cong K_4$. Let $V(H_0)=\{u_1,u_2,u_3,u_4\}$. We distinguish the following two cases to lead a contradiction, respectively.

{\bf Case 1. }$W_{H_0}=\emptyset$.

In this case, it is easy to get $x_{u_1}=x_{u_2}=x_{u_3}=x_{u_4}$. By $\lambda x_{u_1}=x_{u_2}+x_{u_3}+x_{u_4}+x_{u^\ast}$, we obtain $x_{u_1}=\frac{x_{u^\ast}}{\lambda-3}$. By (2) and Lemma \ref{component lemma}, we have
\begin{align*}
e(W)&\leq\sum_{H\in\mathcal{H}\setminus H_0}\gamma(H)+\gamma(H_0)+1\\
&\leq0+\frac{2(x_{u_1}+x_{u_2}+x_{u_3}+x_{u_4})}{x_{u^\ast}}-6+1\\
&\leq\frac{8}{\lambda-3}-5.
\end{align*}
Since $\lambda>\frac{49}{5}$, we get $e(W)<0$, a contradiction.

{\bf Case 2. }$W_{H_0}\neq\emptyset$.

Since $G^\ast$ contains no $F_6$, we have $|N_U(w)\cap V(H_0)|=1$ for every $w\in W_{H_0}$. Suppose $x_{u_1}\geq x_{u_2}\geq x_{u_3}\geq x_{u_4}$. Then $N_U(w)\cap V(H_0)=\{u_1\}$ for every $w\in W_{H_0}$. Otherwise, assume $N_U(w_0)\cap V(H_0)=\{u_2\}$ for some $w_0\in W_{H_0}$. Combining with $e(W)\leq1$, we can verify that $G'=G^\ast-w_0u_2+w_0u_1$ is $F_6$-free. By Lemma \ref{lem:Wu-Xiao-2005}, we have $\lambda(G')>\lambda(G^\ast)$, which contradicts the maximality of $G^\ast$. Therefore, $N_W(u_i)=\emptyset$ for $i\in\{2,3,4\}$. By $\lambda x_{u_2}=x_{u_1}+x_{u_3}+x_{u_4}+x_{u^\ast}\leq x_{u_1}+2x_{u_2}+x_{u^\ast}$, we get $x_{u_2}\leq\frac{2x_{u^\ast}}{\lambda-2}$. Note that $x_{u_4}\leq x_{u_3}\leq x_{u_2}$. It follows that $x_{u_4}\leq x_{u_3}\leq\frac{2x_{u^\ast}}{\lambda-2}$.  Hence,
\begin{align*}
e(W)&\leq\sum_{H\in\mathcal{H}\setminus H_0}\gamma(H)+\gamma(H_0)+1\\
&\leq0+\frac{2(x_{u_1}+x_{u_2}+x_{u_3}+x_{u_4})}{x_{u^\ast}}-6+1\\
&\leq2\left(1+3\cdot\frac{2}{\lambda-2}\right)-5\\
&\leq\frac{12}{\lambda-2}-3<0.
\end{align*}
a contradiction. This completes the proof.
\end{proof}

\begin{lem}\label{no K4-e lemma}
For any $H\in\mathcal{H}$, we have $H\ncong K_4-e$.
\end{lem}

\begin{proof}
Suppose to the contrary that $H_1$ is a component in $\mathcal{H}$ such that $H_1\cong K_4-e$. Let $V(H_1)=\{u_1,u_2,u_3,u_4\}$ with $d_{H_1}(u_2)=d_{H_1}(u_4)=3$. We discuss the following two cases to lead a contradiction, respectively.

{\bf Case 1. }$N_W(u_2)\cup N_W(u_4)=\emptyset$.

In this case, it is easy to get $x_{u_2}=x_{u_4}$. By $\lambda x_{u_2}=x_{u_1}+x_{u_3}+x_{u_4}+x_{u^\ast}$, we obtain $x_{u_2}\leq\frac{3x_{u^\ast}}{\lambda-1}$. By (2) and Lemma \ref{component lemma}, we have
\begin{align*}
e(W)&\leq\sum_{H\in\mathcal{H}\setminus H_1}\gamma(H)+\gamma(H_1)+1\\
&\leq0+\frac{x_{u_1}+2x_{u_2}+x_{u_3}+2x_{u_4}}{x_{u^\ast}}-5+1\\
&\leq2+4\cdot\frac{3}{\lambda-1}-4\\
&\leq\frac{12}{\lambda-1}-2.
\end{align*}
Note that $\lambda>\frac{49}{5}$, we have $e(W)<0$, a contradiction.

{\bf Case 2. }$N_W(u_2)\cup N_W(u_4)\neq\emptyset$.

If $\gamma(H_1)<-1$, then by (2) and Lemma \ref{component lemma}, we know that $e(W)<0$, a contradiction. So $-1\leq\gamma(H_1)\leq0$. Since $G^\ast$ contains no $F_6$, we have $|N_U(w)\cap \{u_2,u_4\}|=1$ and $|N_U(w)\cap \{u_1,u_3\}|=0$ for any $w\in N_W(u_2)\cup N_W(u_4)$. Suppose $x_{u_2}\geq x_{u_4}$. Then $N_U(w)\cap V(H_1)=\{u_2\}$ for any $w\in N_W(u_2)\cup N_W(u_4)$. Otherwise, assume $N_U(w_0)\cap V(H_1)=\{u_4\}$ for some $w_0\in N_W(u_2)\cup N_W(u_4)$ . Note that $e(W)\leq1$, it is easy to find that $G'=G^\ast-w_0u_4+w_0u_2$ is $F_6$-free. By Lemma \ref{lem:Wu-Xiao-2005}, we have $\lambda(G')>\lambda(G^\ast)$, contradicting the maximality of $G^\ast$. Hence, $N_W(u_4)=\emptyset$. By $\lambda x_{u_4}=x_{u_1}+x_{u_2}+x_{u_3}+x_{u^\ast}\leq4x_{u^\ast}$, we have $x_{u_4}\leq\frac{4x_{u^\ast}}{\lambda}$.

We first show that if $W_0\neq\emptyset$, then $x_w\leq\frac{\lambda-1}{\lambda}x_{u^\ast}$ for any $w\in W_0$. Suppose that there exists a vertex $w_0\in W_0$ such that $x_{w_0}>\frac{\lambda-1}{\lambda}x_{u^\ast}$. If $w_0\in N_W(u_2)$, then $u_1,u_3,u_4\notin N_U(w_0)$. Therefore, $\lambda x_{w_0}\leq \lambda x_{u^\ast}-x_{u_1}-x_{u_3}-x_{u_4}$. That is, $x_{u_1}+x_{u_3}+x_{u_4}\leq\lambda x_{u^\ast}-\lambda x_{w_0}<x_{u^\ast}$. Hence, $\gamma(H_1)=\frac{x_{u_1}+2x_{u_2}+x_{u_3}+2x_{u_4}}{x_{u^\ast}}-5<1+3-5=-1$, a contradiction. If $w_0\notin N_W(u_2)$, then by $N_W(u_4)=\emptyset$, we have $u_2,u_4\notin N_U(w_0)$. So $\lambda x_{w_0}\leq \lambda x_{u^\ast}-x_{u_2}-x_{u_4}$. Thus, $x_{u_2}+x_{u_4}<x_{u^\ast}$. It follows that $\gamma(H_1)=\frac{x_{u_1}+2x_{u_2}+x_{u_3}+2x_{u_4}}{x_{u^\ast}}-5<2+2-5=-1$, a contradiction. Therefore, $x_w\leq\frac{\lambda-1}{\lambda}x_{u^\ast}$ for any $w\in W_0$.

Since $G^\ast$ is $F_6$-free, we have $N_W(u_1)\cap N_W(u_2)=\emptyset$. Then $\sum_{i=1}^2\sum_{w\in N_{W_{H_1}}(u_i)}x_w=\sum_{w\in W_0\cap \left(\cup_{i\in\{1,2\}}N_W(u_i)\right)}x_w+\sum_{w\in \left(\cup_{i\in\{1,2\}}N_W(u_i)\right)\setminus W_0}x_w$. Since $x_w\leq\frac{\lambda-1}{\lambda}x_{u^\ast}$ for any $w\in W_0$, we have $\sum_{w\in W_0\cap \left(\cup_{i\in\{1,2\}}N_W(u_i)\right)}x_w\leq \frac{(\lambda-1)e(H_1,W_0)}{\lambda}x_{u^\ast}$. Since $d_W(w)\geq1$ for $w\in (\cup_{i\in\{1,2\}}N_W(u_i))\setminus W_0$, we have $\sum_{w\in \left(\cup_{i\in\{1,2\}}N_W(u_i)\right)\setminus W_0}x_w\leq\sum_{w\in W_{H_1}\setminus W_0}d_W(w)x_w\leq2e(W)x_{u^\ast}$. By
\begin{align*}
\begin{cases}
\lambda x_{u_1}= x_{u_2}+x_{u_4}+x_{u^\ast}+\sum_{w\in N_{W_{H_1}}(u_1)}x_w,\\
\lambda x_{u_2}= x_{u_1}+x_{u_3}+x_{u_4}+x_{u^\ast}+\sum_{w\in N_{W_{H_1}}(u_2)}x_w,
\end{cases}
\end{align*}
we obtain $\lambda(x_{u_1}+x_{u_2})=x_{u_1}+x_{u_2}+x_{u_3}+2x_{u_4}+2x_{u^\ast}+\sum_{i=1}^2\sum_{w\in N_{W_{H_1}}(u_i)}x_w$. That is, \begin{align*}
x_{u_1}+x_{u_2}&=\frac{x_{u_3}+2x_{u_4}+2x_{u^\ast}+\sum_{i=1}^2\sum_{w\in N_{W_{H_1}}(u_i)}x_w}{\lambda-1}\\
&\leq\frac{x_{u^\ast}+2\cdot\frac{4x_{u^\ast}}{\lambda}+2x_{u^\ast}+\frac{(\lambda-1)e(H_1,W_0)}{\lambda}x_{u^\ast}+2e(W)x_{u^\ast}}{\lambda-1}\\
&=\frac{(3+\frac{8}{\lambda})x_{u^\ast}+\frac{(\lambda-1)e(H_1,W_0)}{\lambda}x_{u^\ast}+2e(W)x_{u^\ast}}{\lambda-1}.
\end{align*}
Furthermore, by \ref{component lemma}, we know that $\gamma(H)=\sum_{u\in V(H)}(d_H(u)-1)\frac{x_u}{x_{u^\ast}}-e(H)\leq0$. That is, $\sum_{u\in V(H)}(d_H(u)-1)\frac{x_u}{x_{u^\ast}}\leq e(H)$ for any $H\in \mathcal{H}\setminus H_1$.
Thus,
\begin{align*}
&\sum_{u\in U_+}(d_U(u)-1)\frac{x_u}{x_{u^\ast}}+\sum_{w\in W}d_U(w)\frac{x_w}{x_{u^\ast}}\\
&=\sum_{H\in \mathcal{H}\setminus H_1}\sum_{u\in V(H)}(d_H(u)-1)\frac{x_u}{x_{u^\ast}}+\sum_{u\in V(H_1)}(d_{H_1}(u)-1)\frac{x_u}{x_{u^\ast}}+\sum_{w\in W_0}d_U(w)\frac{x_w}{x_{u^\ast}}+\sum_{w\in W\setminus W_0}d_U(w)\frac{x_w}{x_{u^\ast}}\\
&\leq\sum_{H\in \mathcal{H}\setminus H_1}e(H)+\frac{3+\frac{8}{\lambda}+\frac{(\lambda-1)e(H_1,W_0)}{\lambda}+2e(W)}{\lambda-1}+2+2\cdot\frac{4}{\lambda}+\frac{(\lambda-1)e(U,W_0)}{\lambda}+e(U,W\setminus W_0)\\
&=\sum_{H\in \mathcal{H}\setminus H_1}e(H)+\frac{3+\frac{8}{\lambda}}{\lambda-1}+2+2\cdot\frac{4}{\lambda}+\frac{e(H_1,W_0)}{\lambda}+\frac{(\lambda-1)e(U,W_0)}{\lambda}+\frac{2e(W)}{\lambda-1}+e(U,W\setminus W_0).
\end{align*}
Note that $\lambda>\frac{49}{5}$, it can find that $\frac{3+\frac{8}{\lambda}}{\lambda-1}+2+2\cdot\frac{4}{\lambda}=\frac{11}{\lambda-1}+2<4=e(H_1)-1$. Hence,
\begin{align*}
&\sum_{u\in U_+}(d_U(u)-1)\frac{x_u}{x_{u^\ast}}+\sum_{w\in W}d_U(w)\frac{x_w}{x_{u^\ast}}\\
&<\sum_{H\in \mathcal{H}\setminus H_1}e(H)+e(H_1)-1+e(U,W_0)+e(W)+e(U,W\setminus W_0)\\
&=e(U_+)-1+e(U,W)+e(W),
\end{align*}
which contradicts (1). This completes the proof.
\end{proof}
\begin{lem}\label{edge number of W}
$e(W)=0$.
\end{lem}

\begin{proof}
Suppose that $e(W)\neq0$. Recall that $e(W)\leq1$. Then $e(W)=1$. Combining Lemma \ref{component lemma} and inequality (2), we obtain $0\geq\sum_{H\in\mathcal{H}}\gamma(H)\geq\sum_{u\in U_0}\frac{x_u}{x_{u^\ast}}$. Since $x_u>0$ for any $u\in V(G^\ast)$, we get $U_0=\emptyset$ and $\gamma(H)=0$ for any $H\in\mathcal{H}$. Moreover, $x_w=x_{u^\ast}$ for $w\in W$ and $d_U(w)\geq1$. Let $E(W)=\{w_1w_2\}$. By Lemma \ref{lem:Zhai-Lin-2021}, we know that $d_U(w_1)\geq1$ and $d_U(w_2)\geq1$. So $x_{w_1}=x_{w_2}=x_{u^\ast}$. Note that $\mathcal{H}$ is not empty. Choose a component $H\in\mathcal{H}$. By Lemmas \ref{component cases}, \ref{upperbound}, \ref{no K4 lemma} and \ref{no K4-e lemma}, we have $H\cong K_{1,r}+e$ or $C_4$. If $H\cong K_{1,r}+e$ with $r\geq2$, then $H$ contains a triangle $C_3$. Let $V(C_3)=\{u_1,u_2,u_3\}$. It follows that $w_i$ is adjacent to at least two vertices of $C_3$ for $i\in\{1,2\}$. Otherwise, assume $w_1$ is adjacent to at most one vertex, without loss of generality, suppose $u_1$. Then $\lambda x_{w_1}\leq x_{w_2}+\lambda x_{u^\ast}-x_{u_2}-x_{u_3}$. Since $\gamma(H)=0$, we know that $x_{u_i}=x_{u^\ast}$ for $i\in\{1,2,3\}$. So $\lambda x_{w_1}\leq (\lambda-1) x_{u^\ast}$, which contradicts $x_{w_1}=x_{u^\ast}$. Therefore, $w_i$ is adjacent to at least two vertices of $C_3$ for $i\in\{1,2\}$. One can find that $G^\ast$ contains an $F_6$, a contradiction. If $H\cong C_4$, similarly, $w_i$ is adjacent to at least three vertices of $C_4$ for $i\in\{1,2\}$. One can verify that $G^\ast$ contains an $F_6$, a contradiction. Hence, $e(W)=0$.
\end{proof}

\begin{lem}\label{no C4 lemma}
For any $H\in\mathcal{H}$, we have $H\ncong C_4$.
\end{lem}

\begin{proof}
Suppose that $G^\ast[U]$ contains a component $H_2\in\mathcal{H}$ satisfying $H_2\cong C_4$. If $\gamma(H_2)<-1$, combining Lemma \ref{component lemma} and (2), we have $e(W)<0$, a contradiction. So $-1\leq\gamma(H_2)\leq0$. We proceed by distinguishing the following two cases.

{\bf Case 1. }$|W_{H_2}|\leq1$.

For any $u\in V(H_2)$, we have $d_{G^\ast}(u)\leq4$. Then $\lambda x_u=\sum_{v\in N(u)}x_v\leq4x_{u^\ast}$. That is, $x_u\leq\frac{4x_{u^\ast}}{\lambda}$. Therefore, $\gamma(H_2)=\sum_{u\in V(H_2)}\frac{x_u}{x_{u^\ast}}-4\leq\frac{16}{\lambda}-4<-1$, a contradiction.

{\bf Case 2. }$|W_{H_2}|\geq2$.

Let $H=u_1u_2u_3u_4u_1$. We first show the following two claims.

\begin{claim}\label{result2.1}
There is at most one vertex $w$ in $W_{H_2}$ such that $d_{H_2}(w)\geq3$.
\end{claim}

\begin{proof}
Suppose to the contrary that there are two vertices $w_1$ and $w_2$ in $W_{H_2}$ satisfying $d_{H_2}(w_1)\geq3$ and $d_{H_2}(w_2)\geq3$. Then $|N_{H_2}(w_1)\cap N_{H_2}(w_2)|\geq2$. If $|N_{H_2}(w_1)\cap N_{H_2}(w_2)|\geq3$, without loss of generality, assume $\{u_1,u_2,u_3\}\subseteq N_{H_2}(w_1)\cap N_{H_2}(w_2)$, then $u_2\vee \{w_1u_1u^\ast u_3w_2\}$ is an $F_6$, a contradiction. If $|N_{H_2}(w_1)\cap N_{H_2}(w_2)|=2$, let $N_{H_2}(w_1)\cap N_{H_2}(w_2)=\{u_i,u_j\}$ with $i\neq j\in\{1,2,3,4\}$. If $u_i$ and $u_j$ are adjacent, without loss of generality, assume $\{u_i,u_j\}=\{u_1,u_2\}$. Since $d_{H_2}(w_1)\geq3$, we get that $u_3$ or $u_4$ is a neighbor of $w_1$. Suppose that $u_3$ is a neighbor of $w_1$. Because $d_{H_2}(w_2)\geq3$ and $|N_{H_2}(w_1)\cap N_{H_2}(w_2)|=2$, we know that $u_4$ is a neighbor of $w_2$. We can observe that $u_1\vee\{w_1u_2u^\ast u_4w_2\}$ is an $F_6$, a contradiction. If $u_i$ and $u_j$ are not adjacent, without loss of generality, assume $\{u_i,u_j\}=\{u_1,u_3\}$. Since $d_{H_2}(w_i)\geq3$ for $i\in\{1,2\}$ and $|N_{H_2}(w_1)\cap N_{H_2}(w_2)|=2$, suppose that $u_2$ is a neighbor of $w_1$ and $u_4$ is a neighbor of $w_2$. We can find that $u_1\vee\{w_1u_2u^\ast u_4w_2\}$ is an $F_6$, a contradiction. The proof is complete.
\end{proof}

\begin{claim}\label{result2.2}
For any $w\in W$ satisfying $d_{H_2}(w)\leq2$, we have $x_w\leq\frac{\lambda-1}{\lambda}x_{u^\ast}$.
\end{claim}

\begin{proof}
Suppose that there exists a $w_0\in W$ such that $d_{H_2}(w_0)\leq2$ and $x_{w_0}>\frac{\lambda-1}{\lambda}x_{u^\ast}$. Since $d_{H_2}(w_0)\leq2$, there are at least two vertices, denoted by $u_i$ and $u_j$ with $i\neq j\in\{1,2,3,4\}$, in $V(H_2)$ that are not neighbors of $w_0$. Then $\lambda x_{w_0}\leq\lambda x_{u^\ast}-x_{u_i}-x_{u_j}$. That is, $x_{u_i}+x_{u_j}\leq\lambda x_{u^\ast}-\lambda x_{w_0}<x_{u^\ast}$. Therefore, $\gamma(H_2)=\sum_{u\in V(H_2)}\frac{x_u}{x_{u^\ast}}-4<2+1-4=-1$, which contradicts $-1\leq\gamma(H_2)\leq0$. The proof is complete.
\end{proof}

Now we come back to show Lemma \ref{no C4 lemma}. By
\begin{align*}
\begin{cases}
\lambda x_{u_1}=x_{u_2}+x_{u_4}+x_{u^\ast}+\sum_{w\in N_{W_{H_2}}(u_1)}x_w,\\
\lambda x_{u_2}=x_{u_1}+x_{u_3}+x_{u^\ast}+\sum_{w\in N_{W_{H_2}}(u_2)}x_w,\\
\lambda x_{u_3}=x_{u_2}+x_{u_4}+x_{u^\ast}+\sum_{w\in N_{W_{H_2}}(u_3)}x_w,\\
\lambda x_{u_4}=x_{u_1}+x_{u_3}+x_{u^\ast}+\sum_{w\in N_{W_{H_2}}(u_4)}x_w,
\end{cases}
\end{align*}
we have $(\lambda-1)(x_{u_1}+x_{u_2}+x_{u_3}+x_{u_4})\leq8x_{u^\ast}+\sum_{i=1}^4\sum_{w\in N_{W_{H_2}}(u_i)}x_w$. That is, $(x_{u_1}+x_{u_2}+x_{u_3}+x_{u_4})\leq\frac{8x_{u^\ast}+\sum_{i=1}^4\sum_{w\in N_{W_{H_2}}(u_i)}x_w}{\lambda-1}$. By Claims \ref{result2.1} and \ref{result2.2}, we obtain $\sum_{i=1}^4\sum_{w\in N_{W_{H_2}}(u_i)}x_w\leq4x_{u^\ast}+(e(H_2,W)-3)\cdot\frac{\lambda-1}{\lambda}x_{u^\ast}$ and $\sum_{w\in W}d_U(w)\frac{x_w}{x_{u^\ast}}\leq4+(e(H_2,W)-3)\cdot\frac{\lambda-1}{\lambda}+e(U\setminus H_2,W)$. Recall that $\gamma(H)\leq0$. It follows that $\sum_{u\in V(H)}(d_H(u)-1)\frac{x_u}{x_{u^\ast}}\leq e(H)$ for any $H\in \mathcal{H}\setminus H_2$. Thus, by $\lambda>\frac{49}{5}$, we obtain
\begin{align*}
&\sum_{u\in U_+}(d_U(u)-1)\frac{x_u}{x_{u^\ast}}+\sum_{w\in W}d_U(w)\frac{x_w}{x_{u^\ast}}\\
&=\sum_{H\in \mathcal{H}\setminus H_2}\sum_{u\in V(H)}(d_H(u)-1)\frac{x_u}{x_{u^\ast}}+\sum_{u\in V(H_2)}\frac{x_u}{x_{u^\ast}}+\sum_{w\in W}d_U(w)\frac{x_w}{x_{u^\ast}}\\
&\leq\sum_{H\in \mathcal{H}\setminus H_2}e(H)+\frac{8x_{u^\ast}+\sum_{i=1}^4\sum_{w\in N_{W_{H_2}}(u_i)}x_w}{(\lambda-1)x_{u^\ast}}+\sum_{w\in W}d_U(w)\frac{x_w}{x_{u^\ast}}\\
&\leq\sum_{H\in \mathcal{H}\setminus H_2}e(H)+\frac{8}{\lambda-1}+\frac{4+(e(H_2,W)-3)\cdot\frac{\lambda-1}{\lambda}}{\lambda-1}+4+(e(H_2,W)-3)\cdot\frac{\lambda-1}{\lambda}+e(U\setminus H_2,W)\\ &=e(U_+)-4+\frac{8}{\lambda-1}+\frac{4}{\lambda-1}+\frac{e(H_2,W)}{\lambda}-\frac{3}{\lambda}+4+\frac{(\lambda-1)e(H_2,W)}{\lambda}-\frac{3(\lambda-1)}{\lambda}+e(U\setminus H_2,W)\\
&=e(U_+)+\frac{12}{\lambda-1}+e(H_2,W)-3+e(U\setminus H_2,W)\\
&<e(U_+)+e(U,W)+e(W)-1,
\end{align*}
which contradicts (1). This completes the proof.
\end{proof}

\begin{lem}\label{no K1re lemma}
For any $H\in\mathcal{H}$, we have $H\ncong K_{1,r}+e$ where $r\geq2$.
\end{lem}

\begin{proof}
Suppose to the contrary that $G^\ast[U]$ contains a component $H_3\cong K_{1,r}+e$ for some $r\geq2$. Then $H_3$ contains a triangle $C_3$. Let $V(C_3)=\{u_1,u_2,u_3\}$ with $d_{H_3}(u_1)=d_{H_3}(u_2)=2$. We proceed by considering the following two possible cases.

{\bf Case 1. }$|W_{H_3}|\leq1$.

By
\begin{align*}
\begin{cases}
\lambda x_{u_1}\leq x_{u_2}+x_{u_3}+x_{u^\ast}+x_{u^\ast},\\
\lambda x_{u_2}\leq x_{u_1}+x_{u_3}+x_{u^\ast}+x_{u^\ast},
\end{cases}
\end{align*}
we get $x_{u_1}\leq \frac{4x_{u^\ast}}{\lambda}$ and $x_{u_2}\leq \frac{4x_{u^\ast}}{\lambda}$. Therefore, $\gamma(H_3)=\frac{x_{u_1}}{x_{u^\ast}}+\frac{x_{u_2}}{x_{u^\ast}}+(r-1)\frac{x_{u_3}}{x_{u^\ast}}-(r+1)\leq \frac{8}{\lambda}-2$. Note that $\lambda>\frac{49}{5}$. It is easy to get $\gamma(H_3)<-1$ and $e(W)<0$, a contradiction.

{\bf Case 2. }$|W_{H_3}|\geq2$.

We first show that $d_{C_3}(w)\leq2$ for any $w\in W$. Otherwise, there exists a vertex $w_0\in W$ such that $d_{C_3}(w_0)=3$. If $r\geq3$, then $G^\ast$ contains an $F_6$, a contradiction. So $r=2$. That is, $H_3\cong C_3$. If there exists another vertex $w_1$ satisfying $d_{C_3}(w_1)\geq2$, without loss of generality, assume $u_1,u_2\in N_{C_3}(w_1)$. Then $u_1\vee\{w_0u_3u^\ast u_2w_1\}$ is an $F_6$, a contradiction. Thus, $d_{C_3}(w)\leq1$ for any $w\in W\setminus w_0$. Suppose $x_{u_1}\geq x_{u_2}\geq x_{u_3}$. We have $N_{C_3}(w)\cap V(H_3)=\{u_1\}$ or $\emptyset$ for any $w\in W\setminus w_0$. Otherwise, assume $N_{C_3}(w)\cap V(H_3)=\{u_i\}$ where $i\in\{2,3\}$, then $G'=G^\ast-wu_i+wu_1$ is $F_6$-free. By Lemma \ref{lem:Wu-Xiao-2005}, we know that $\lambda(G')>\lambda(G^\ast)$, which contradicts the maximality of $G^\ast$. Hence, $N_W(u_2)=N_W(u_3)=\{w_0\}$. It follows that $\lambda x_{u_2}=x_{u_1}+x_{u_3}+x_{u^\ast}+x_{w_0}\leq4x_{u^\ast}$ and $\lambda x_{u_3}=x_{u_1}+x_{u_2}+x_{u^\ast}+x_{w_0}\leq4x_{u^\ast}$. Equivalently, $x_{u_2}\leq\frac{4}{\lambda}x_{u^\ast}$ and $x_{u_3}\leq\frac{4}{\lambda}x_{u^\ast}$. Thus, by $\lambda>\frac{49}{5}$, we obtain $\gamma(H_3)=\sum_{i=1}^3\frac{x_{u_i}}{x_{u^\ast}}-3\leq1+\frac{8}{\lambda}-3<-1$ and $e(W)<0$, a contradiction. Therefore, $d_{C_3}(w)\leq2$ for any $w\in W$.

If there exists another non-trivial component $H'$, then $H'\cong K_{1,r'}+e$ or $K_{1,p}$ or $D_{a,b}$ where $r'\geq2$ and $p,a,b\geq1$. If $H'\cong K_{1,p}$ or $D_{a,b}$, by Lemma \ref{upperbound}, we know $\gamma(H')\leq-1$. Since $d_{C_3}(w)\leq2$ for any $w\in W$ and Lemma \ref{edge number of W}, we have $N_{G^\ast}(w)\subset N_{G^\ast}(u^\ast)$. Furthermore, $\lambda x_w=\sum_{u\in N_{G^\ast}(w)}x_u<\sum_{u\in N_{G^\ast}(u^\ast)}x_u=\lambda x_{u^\ast}$. That is, $x_w<x_{u^\ast}$ for any $w\in W$. Therefore, $\sum_{w\in W}d_U(w)\frac{x_w}{x_{u^\ast}}<e(U,W)$. By the first inequality of (2), we have $e(W)<\sum_{H\in\mathcal{H}}\gamma(H)-\sum_{u\in U_0}\frac{x_u}{x_{u^\ast}}+1\leq-1-\sum_{u\in U_0}\frac{x_u}{x_{u^\ast}}+1\leq0$, a contradiction. Next we consider $H'\cong K_{1,r'}+e$ with $r'\geq2$. Let $C'_3$ is a triangle in $K_{1,r'}+e$ with $V(C'_3)=\{v_1,v_2,v_3\}$ and $d_{H'}(v_1)=d_{H'}(v_2)=2$. Similarly, $d_{C'_3}(w)\leq2$ for any $w\in W$. Now we have $x_w\leq\frac{\lambda-1}{\lambda}x_{u^\ast}$ for any $w\in W$. Otherwise, there is a vertex $w_0\in W$ such that $x_{w_0}>\frac{\lambda-1}{\lambda}x_{u^\ast}$. Since $d_{C_3}(w_0)\leq2$ and $d_{C'_3}(w_0)\leq2$, It follows that there are $u_i\in V(C_3)$ and $v_j\in V(C'_3)$ which are not neighbors of $w_0$ for some $i,j\in\{1,2,3\}$. Because $e(W)=0$, we obtain $\lambda x_{w_0}\leq\lambda x_{u^\ast}-x_{u_i}-x_{v_j}$. Equivalently, $x_{u_i}+x_{v_j}\leq\lambda x_{u^\ast}-\lambda x_{w_0}<x_{u^\ast}$. Thus
\begin{align*}
\gamma(H_3)+\gamma(H')&=\sum_{u\in V(H_3)}(d_{H_3}(u)-1)\frac{x_u}{x_{u^\ast}}-\frac{x_{u_i}}{x_{u^\ast}}+\frac{x_{u_i}}{x_{u^\ast}}-e(H_3)\\
&+\sum_{u\in V(H')}(d_{H'}(u)-1)\frac{x_u}{x_{u^\ast}}-\frac{x_{v_j}}{x_{u^\ast}}+\frac{x_{v_j}}{x_{u^\ast}}-e(H')\\
&=\frac{x_{u_1}}{x_{u^\ast}}+\frac{x_{u_2}}{x_{u^\ast}}+(r-1)\frac{x_{u_3}}{x_{u^\ast}}-\frac{x_{u_i}}{x_{u^\ast}}+\frac{x_{u_i}}{x_{u^\ast}}-(r+1)\\
&+\frac{x_{v_1}}{x_{u^\ast}}+\frac{x_{v_2}}{x_{u^\ast}}+(r'-1)\frac{x_{v_3}}{x_{u^\ast}}-\frac{x_{v_j}}{x_{u^\ast}}+\frac{x_{v_j}}{x_{u^\ast}}-(r'+1)\\
&\leq r+\frac{x_{u_i}}{x_{u^\ast}}-(r+1)+r'+\frac{x_{v_j}}{x_{u^\ast}}-(r'+1)\\
&<r-(r+1)+r'-(r'+1)+1=-1.
\end{align*}
By (2) and Lemma \ref{component lemma}, we find that $e(W)\leq\sum_{H\in\mathcal{H}}\gamma(H)-\sum_{u\in U_0}\frac{x_u}{x_{u^\ast}}+1<-1+0+1=0$, a contradiction. Therefore, $x_w\leq\frac{\lambda-1}{\lambda}x_{u^\ast}$ for any $w\in W$. Furthermore, $\sum_{w\in W}d_U(w)\frac{x_w}{x_{u^\ast}}\leq\frac{\lambda-1}{\lambda}e(U,W)$. Since
\begin{align*}
\begin{cases}
\lambda x_{u_1}=x_{u_2}+x_{u_3}+x_{u^\ast}+\sum_{w\in N_{W}(u_1)}x_w,\\
\lambda x_{u_2}=x_{u_1}+x_{u_3}+x_{u^\ast}+\sum_{w\in N_{W}(u_2)}x_w,
\end{cases}
\end{align*}
we have
\begin{align*}
x_{u_1}+x_{u_2}&\leq\frac{4x_{u^\ast}+\sum_{i=1}^2\sum_{w\in N_{W}(u_i)}x_w}{\lambda-1}\\
&\leq\frac{4x_{u^\ast}+\frac{\lambda-1}{\lambda}e(H_3,W)x_{u^\ast}}{\lambda-1}\\
&=\frac{4x_{u^\ast}}{\lambda-1}+\frac{e(H_3,W)x_{u^\ast}}{\lambda}.
\end{align*}
Hence, by $\lambda>\frac{49}{5}$, we have
\begin{align*}
&\sum_{u\in U_+}(d_U(u)-1)\frac{x_u}{x_{u^\ast}}+\sum_{w\in W}d_U(w)\frac{x_w}{x_{u^\ast}}\\
&=\sum_{H\in \mathcal{H}\setminus H_3}\sum_{u\in V(H)}(d_H(u)-1)\frac{x_u}{x_{u^\ast}}+\sum_{u\in V(H_3)}(d_{H_3}(u)-1)\frac{x_u}{x_{u^\ast}}+\sum_{w\in W}d_U(w)\frac{x_w}{x_{u^\ast}}\\
&\leq\sum_{H\in \mathcal{H}\setminus H_3}e(H)+\frac{4}{\lambda-1}+\frac{e(H_3,W)}{\lambda}+(r-1)+\frac{\lambda-1}{\lambda}e(U,W)\\
&\leq e(U_+)-(r+1)+\frac{4}{\lambda-1}+e(U,W)+(r-1)\\
&=e(U_+)+e(U,W)+\frac{4}{\lambda-1}-2\\
&<e(U_+)+e(U,W)+e(W)-1,
\end{align*}
which contradicts (1). Thus there is exactly one non-trivial component $H_3$. By Lemma \ref{edge number of U}, we know that $r\geq3$.

By (2), we have $0=e(W)\leq\gamma(H_3)-\sum_{u\in U_0}\frac{x_u}{x_{u^\ast}}+1$. By Lemma \ref{component lemma}, we get $0\leq0-\sum_{u\in U_0}\frac{x_u}{x_{u^\ast}}+1$. That is, $\sum_{u\in U_0}\frac{x_u}{x_{u^\ast}}\leq1$. Since $\lambda x_{u^\ast}=x_{u_1}+x_{u_2}+x_{u_3}+\sum_{v\in N_{H_3}(u_3)\setminus\{u_1,u_2\}}x_v+\sum_{u\in U_0}x_u$, we have
\begin{align*}
\sum_{v\in N_{H_3}(u_3)\setminus\{u_1,u_2\}}x_v&=\lambda x_{u^\ast}-x_{u_1}-x_{u_2}-x_{u_3}-\sum_{u\in U_0}x_u\\
&\geq(\lambda-1)x_{u^\ast}-x_{u_1}-x_{u_2}-x_{u_3}.
\end{align*}
Therefore,
\begin{align*}
\lambda x_{u_3}&=x_{u_1}+x_{u_2}+x_{u^\ast}+\sum_{v\in N_{H_3}(u_3)\setminus\{u_1,u_2\}}x_v+\sum_{w\in N_W(u_3)}x_w\\
&\geq x_{u_1}+x_{u_2}+x_{u^\ast}+(\lambda-1)x_{u^\ast}-x_{u_1}-x_{u_2}-x_{u_3}+\sum_{w\in N_W(u_3)}x_w\\
&=\lambda x_{u^\ast}-x_{u_3}+\sum_{w\in N_W(u_3)}x_w.
\end{align*}
It follows that $x_{u_3}\geq\frac{\lambda}{\lambda+1}x_{u^\ast}$ and $\sum_{w\in N_W(u_3)}x_w\leq(\lambda+1)x_{u_3}-\lambda x_{u^\ast}\leq x_{u^\ast}$. For any $w\in W$, if $w\in N_W(u_3)$, then $d_{H_3}(w)=1$. Otherwise, recall that $r\geq3$, there is an $F_6$ in $G^\ast$. So $\lambda x_w\leq x_{u_3}+\sum_{u\in U_0}x_u\leq2x_{u^\ast}$. If $w\notin N_W(u_3)$, then $\lambda x_w\leq \lambda x_{u^\ast}-x_{u_3}\leq\lambda x_{u^\ast}-\frac{\lambda}{\lambda+1}x_{u^\ast}=\frac{\lambda^2}{\lambda+1}x_{u^\ast}$. That is, $x_w\leq\frac{\lambda}{\lambda+1}x_{u^\ast}$. Note that $\lambda>\frac{49}{5}$, we have $x_w\leq\frac{\lambda}{\lambda+1}x_{u^\ast}$ for any $w\in W$. By
\begin{align*}
\begin{cases}
\lambda x_{u_1}=x_{u_2}+x_{u_3}+x_{u^\ast}+\sum_{w\in N_{W}(u_1)}x_w,\\
\lambda x_{u_2}=x_{u_1}+x_{u_3}+x_{u^\ast}+\sum_{w\in N_{W}(u_2)}x_w,
\end{cases}
\end{align*}
we have
\begin{align*}
\lambda(x_{u_1}+x_{u_2})&\leq6x_{u^\ast}+\sum_{i=1}^2\sum_{w\in N_{W}(u_i)}x_w\\
&\leq6x_{u^\ast}+e(H_3,W)\cdot\frac{\lambda}{\lambda+1}x_{u^\ast}.
\end{align*}
Equivalently, $x_{u_1}+x_{u_2}\leq\frac{6}{\lambda}x_{u^\ast}+\frac{e(H_3,W)}{\lambda+1}x_{u^\ast}$. Hence, by $\lambda>\frac{49}{5}$, we have
\begin{align*}
&\sum_{u\in U_+}(d_U(u)-1)\frac{x_u}{x_{u^\ast}}+\sum_{w\in W}d_U(w)\frac{x_w}{x_{u^\ast}}\\
&=\sum_{u\in V(H_3)}(d_{H_3}(u)-1)\frac{x_u}{x_{u^\ast}}+\sum_{w\in W}d_U(w)\frac{x_w}{x_{u^\ast}}\\
&\leq\frac{6}{\lambda}+\frac{e(H_3,W)}{\lambda+1}+(r-1)+\frac{\lambda}{\lambda+1}e(U,W)\\
&=\frac{6}{\lambda}+e(H_3,W)+\frac{\lambda}{\lambda+1}e(U\setminus H_3,W)+(r-1)\\
&\leq\frac{6}{\lambda}+e(U,W)+(r-1)\\
&<r+e(U,W)\\
&= e(U_+)+e(U,W)-1,
\end{align*}
which contradicts (1). This completes the proof.
\end{proof}

{\bf Proof of Theorem \ref{main theorem}.} By Lemmas \ref{component cases}, \ref{no K4 lemma}, \ref{no K4-e lemma}, \ref{no C4 lemma},  and \ref{no K1re lemma}, we have that each non-trivial component of $G^\ast[U]$ is a $K_{1,r}$ or $D_{a,b}$ where $r,a,b\geq1$. By (2) and Lemma \ref{upperbound}, we obtain that the non-trivial component number of $G^\ast[U]$ is at most 1. If the non-trivial component number of $G^\ast[U]$ is 0, then $G^\ast$ is bipartite. By Lemma \ref{lem:Nikiforov}, $\lambda\leq\sqrt{m}<\frac{1+\sqrt{4m-3}}{2}$, a contradiction. Hence the non-trivial component number of $G^\ast[U]$ is 1. Let $H$ be the unique non-trivial component of $G^\ast[U]$. Then $H\cong K_{1,r}$ or $D_{a,b}$. Since $\gamma(H)\leq-1$ and $0=e(W)\leq\gamma(H)-\sum_{u\in U_0}\frac{x_u}{x_{u^\ast}}+1$, we have $\gamma(H)=-1$ and $U_0=\emptyset$. By (2), we also have $x_w=x_{u^\ast}$ for any $w\in W$ and $d_U(w)\geq1$.

If $H\cong D_{a,b}$, let $V(D_{a,b})=\{u_1,u_2,u_{11},\cdots,u_{1a},u_{21},\cdots,u_{2b}\}$ with $N_H(u_1)=\{u_2,u_{11},\cdots,u_{1a}\}$ and $N_H(u_2)=\{u_1,u_{21},\cdots,u_{2b}\}$. According to the definition of $\gamma(H)$ and $\gamma(H)=-1$, we have $x_{u_1}=x_{u_2}=x_{u^\ast}$. If $W_H=\emptyset$, then $\lambda x_{u^\ast}=x_{u_1}+x_{u_2}+\sum_{i=1}^ax_{u_{1i}}+\sum_{j=1}^bx_{u_{2j}}=2x_{u^\ast}+\sum_{i=1}^ax_{u_{1i}}+\sum_{j=1}^bx_{u_{2j}}$ and $\lambda x_{u_1}=x_{u_2}+x_{u^\ast}+\sum_{i=1}^ax_{u_{1i}}=2x_{u^\ast}+\sum_{i=1}^ax_{u_{1i}}$. Recall that $x_{u_1}=x_{u^\ast}$. It follows that $\sum_{j=1}^bx_{u_{2j}}=0$. which contradicts $x_u>0$ for any $u\in V(G^\ast)$ and $b\geq1$. If $|W_H|=1$, let $W_H=\{w\}$. By Lemma \ref{lem:Zhai-Lin-2021}, we know $d_U(w)\geq2$. Therefore, $x_w=x_{u^\ast}$. Furthermore, $\lambda x_w=\sum_{u\in N(w)}x_u$ and $\lambda x_{u^\ast}=\sum_{u\in N(u^\ast)}x_u$. By $N(w)\subseteq N(u^\ast)$, we obtain $N(w)=N(u^\ast)$. If $a$ or $b\geq2$, suppose $a\geq2$. It is easy to find $u_1\vee\{u_{11}wu_2u^\ast u_{12}\}$ is an $F_6$, a contradiction. Thus, $a=b=1$ and $m=11$, which contradicts $m\geq88$. If $|W_H|\geq2$, let $w_1,w_2\in W_H$. By Lemma \ref{lem:Zhai-Lin-2021}, we have $d_U(w_1)\geq2$ and $d_U(w_2)\geq2$. Therefore, $x_{w_1}=x_{w_2}=x_{u^\ast}$. Furthermore, $N(w_1)=N(w_2)=N(u^\ast)$. Observed that $u_1\vee\{w_1u_{11}u^\ast u_2w_2\}$ is an $F_6$, a contradiction. Hence $H\ncong D_{a,b}$. That is, $H\cong K_{1,r}$.

Let $V(H)=\{u_0,u_1,\cdots,u_r\}$ with the central vertex $u_0$. By the definition of $\gamma(H)$ and $\gamma(H)=-1$, we know $x_{u_0}=x_{u^\ast}$. Since $\lambda x_{u^\ast}=x_{u_0}+\sum_{i=1}^rx_{u_i}=x_{u^\ast}+\sum_{i=1}^rx_{u_i}$ and $\lambda x_{u_0}=x_{u^\ast}+\sum_{i=1}^rx_{u_i}+\sum_{w\in N_W(u_0)}x_w$, we have $\sum_{w\in N_W(u_0)}x_w=0$. Hence, $N_W(u_0)=\emptyset$. If $W\neq\emptyset$, then $d_U(w)\geq2$ for any $w\in W$. So $x_w=x_{u^\ast}$. However, $\lambda x_w=\sum_{u\in N(w)}x_u\leq\sum_{u\in N(u^\ast)}x_u-x_{u_0}<\lambda x_{u^\ast}$. It is a contradiction. Therefore, $W=\emptyset$. Thus, $G^\ast\cong K_1\vee K_{1,r}$ with $2r+1=m$. Equivalently, $G^\ast\cong K_2\vee \frac{m-1}{2}K_1$. This completes the proof.\hfill$\Box$
\section*{\bf Acknowledgments}

This work was supported by National Natural Science Foundation of China (Nos.12131013 and 12161141006).\\

\end{document}